\documentclass{article}
\usepackage[%
journal=???,    
lang=british,   
]{ems-journal}

\usepackage{esint,dsfont}
\usepackage[backgroundcolor=olive, linecolor=olive, textsize=scriptsize]{todonotes}

\usepackage{pgfplots}
\pgfplotsset{width=8cm,compat=1.9}

\def\R{{\mathbb R}}
\def\N{\mathbb{N}}
\def\C{\mathbb{C}}
\def\Z{\mathbb{Z}}

\def\T{\mathbb{T}}
\def\F{\mathbb{F}}

\def\ii{\mathrm{i}}

\newtheorem{prop}{\bf Proposition}[section]
\newtheorem{thm}[prop]{\bf Theorem}
\newtheorem{cor}[prop]{\bf Corollary}
\newtheorem{lem}[prop]{\bf Lemma}
\newtheorem{rmk}[prop]{\it Remark}

\newtheorem{ques}[prop]{\bf Question}

\begin{document}

\title{A connection between Lipschitz and Kazhdan constants for groups of homeomorphisms of the real line}
\titlemark{Lipschitz and Kazhdan constants for groups of homeomorphisms of $\R$}



\emsauthor{1}{
	\givenname{Ignacio}
	\surname{Vergara}
	\mrid{1126379}
	\orcid{0000-0001-7144-4272}}{I.~Vergara}

\Emsaffil{1}{
	\department{Departamento de Matem\'atica y Ciencia de la Computaci\'on}
	\organisation{Universidad de Santiago de Chile}
	\address{Las Sophoras 173}
	\zip{9170020}
	\city{Estaci\'on Central}
	\country{Chile}
	\affemail{ign.vergara.s@gmail.com}}

\classification[37E05, 22F10, 20F60]{22D55}


\begin{abstract}
We exhibit an obstruction for groups with Relative Property (T) to act on the real line by bi-Lipschitz homeomorphisms. This condition is expressed in terms of the Lipschitz and Kazhdan constants associated to finite generating subsets. As an application, we obtain an explicit lower bound for the Lipschitz constants associated to actions of the semidirect product $\F_2\ltimes\Z^2$. We also obtain an upper bound for the Kazhdan constants of pairs of orderable groups, depending only on the cardinal of the generating subset.
\end{abstract}

\maketitle



\section{Introduction}
A finitely generated group $G$ has Property (T) if there is a finite subset $S\subset G$ and a constant $\varepsilon>0$ such that, for every unitary representation $\pi$ of $G$ on a Hilbert space $\mathcal{H}$,
\begin{align*}
\max_{g\in S}\|\pi(g)\xi-\xi\|\geq \varepsilon\|\xi\|\quad\forall\xi\in(\mathcal{H}^\pi)^\perp,
\end{align*}
where $(\mathcal{H}^\pi)^\perp$ denotes the orthogonal complement of the subspace of $\pi$-invariant vectors. We refer the reader to \cite{BedlHVa} for a detailed account on Property (T).

It is an open problem to determine whether a group with Property (T) can act faithfully on $\R$ by orientation-preserving homeomorphisms. Such actions do exist, however, for certain groups satisfying a relative version of Property (T). This is the case of the semidirect product $\F_2\ltimes\Z^2$, which we discuss in extensive detail below.

For a finitely generated group $G$, the existence of a faithful action $G\curvearrowright\R$ by orientation-preserving homeomorphisms is equivalent to the existence of a left-invariant order on $G$; see \cite[\S 1.1.3]{DeNaRi}. Furthermore, it was shown in \cite[Theorem 8.5]{DKNP} that, in this case, there is always an action by bi-Lipschitz homeomorphisms with bounded displacement.

Let $\operatorname{Homeo}_+(\R)$ denote the group of orientation-preserving (i.e. increasing) homeomorphisms of $\R$. We say that $f\in\operatorname{Homeo}_+(\R)$ is bi-Lipschitz if there is a constant $L\geq 1$ such that
\begin{align*}
\frac{1}{L}|x-y|\leq|f(x)-f(y)|\leq L|x-y|\quad\forall x,y\in\R.
\end{align*}
The bi-Lipschitz constant $\operatorname{BiLip}(f)$ is the infimum of all $L\geq 1$ such that the condition above holds.

\begin{rmk}\label{Rmk_Lip_D}
Let $f\in\operatorname{Homeo}_+(\R)$ be a bi-Lipschitz homoeomorphism, and let $Df$ denote its derivative function, which is defined almost everywhere. Then
\begin{align*}
\operatorname{BiLip}(f)=\max\left\{\|Df\|_\infty, \|Df^{-1}\|_\infty\right\},
\end{align*}
where $\|\,\cdot\,\|_\infty$ stands for the essential supremum norm.
\end{rmk}

We say that $f$ has bounded displacement if
\begin{align*}
\operatorname{disp}(f)=\sup_{x\in\R}|f(x)-x|<\infty.
\end{align*}
We let $\operatorname{BiLip}_+^{\mathrm{bd}}(\R)$ denote the subgroup of $\operatorname{Homeo}_+(\R)$ given by all bi-Lipschitz functions with bounded displacement. The fact that it is a group follows from observing that
\begin{align*}
\operatorname{BiLip}(fg)&\leq\operatorname{BiLip}(f)\operatorname{BiLip}(g), & \operatorname{disp}(fg)&\leq\operatorname{disp}(f)+\operatorname{disp}(g),
\end{align*}
and
\begin{align*}
\operatorname{BiLip}(f^{-1})&=\operatorname{BiLip}(f), & \operatorname{disp}(f^{-1})=\operatorname{disp}(f).
\end{align*}
We are interested in finding obstructions for groups with Property (T) to be realised as subgroups of $\operatorname{BiLip}_+^{\mathrm{bd}}(\R)$. The following result seems to be known to experts; we include a proof of it in Section \ref{S_limits} for the sake of completeness.

\begin{prop}\label{Prop_inf_ab_q}
Let $G$ be a group generated by a finite subset $S$. Assume that there is a sequence of injective group homomorphisms $\theta_n:G\to\operatorname{BiLip}_+^{\mathrm{bd}}(\R)$ such that
\begin{align*}
\lim_{n\to\infty}\operatorname{BiLip}(\theta_n(g))=1
\end{align*}
for all $g\in S$. Then there is a nontrivial group homomorphism from $G$ to $(\R,+)$. In particular, $G$ has an infinite abelian quotient.
\end{prop}

Since Property (T) is inherited by quotients, a group satisfying the hypotheses of Proposition \ref{Prop_inf_ab_q} cannot have Property (T); see \cite[Theorem 1.3.4]{BedlHVa} and \cite[Corollary 1.3.6]{BedlHVa}. In other words, if $G$ is a group with Property (T), generated by a finite subset $S$, and $G$ acts on $\R$ by homeomorphisms in $\operatorname{BiLip}_+^{\mathrm{bd}}(\R)$, then
\begin{align*}
\max_{g\in S}\operatorname{BiLip}(g)
\end{align*}
cannot be arbitrarily close to $1$. The main goal of this paper is to obtain a quantitative version of this fact, which will be expressed in terms of the Kazhdan constant of $S$.

We will work in the more general framework of Relative Property (T). Let $G$ be a finitely generated group, and let $\Gamma$ be a subgroup of $G$. We say that the pair $(G,\Gamma)$ has Property (T) if, for every finite generating set $S\subset G$, there is $\varepsilon>0$ such that, for every unitary representation $\pi$ of $G$ on a Hilbert space $\mathcal{H}$ without nontrivial $\Gamma$-invariant vectors,
\begin{align*}
\max_{g\in S}\|\pi(g)\xi-\xi\|\geq \varepsilon\|\xi\|\quad\forall\xi\in\mathcal{H}.
\end{align*}
We refer the reader to \cite[\S 1]{Bur} and \cite{Cor} for details. In the situation described above, we say that $(S,\varepsilon)$ is a Kazhdan pair for $(G,\Gamma)$. For every finite subset $S\subset G$, we define
\begin{align}\label{Kazh_const}
\kappa(G,\Gamma,S)=\sup\left\{\varepsilon\geq 0\ \mid\ (S,\varepsilon)\text{ is a Kazhdan pair for } (G,\Gamma)\right\}.
\end{align}
We call $\kappa(G,\Gamma,S)$ the Kazhdan constant associated to $(G,\Gamma,S)$. Observe that $(G,\Gamma)$ has Property (T) if and only if $\kappa(G,\Gamma,S)>0$ for every finite generating set $S\subset G$. When $\Gamma=G$, we recover the original definition of Property (T), and we define
\begin{align*}
\kappa(G,S)=\kappa(G,G,S).
\end{align*}
By looking at the left regular representation on $\ell^2(G)$, which is given by
\begin{align*}
\lambda(g)\xi(h)=\xi(g^{-1}h)\quad\forall g,h\in G\ \ \forall \xi\in\ell^2(G),
\end{align*}
one sees that $\kappa(G,\Gamma,S)$ cannot be greater than $\sqrt{2}$. Indeed, let $\delta_e$ denote the delta function at the identity element of $G$. Then, for every $g\neq e$,
\begin{align*}
\|\lambda(g)\delta_e-\delta_e\|=\sqrt{2}.
\end{align*}

Let now $G$ be a subgroup of $\operatorname{Homeo}_+(\R)$. We say that $G$ has a global fixed point if there exists $x_0\in\R$ such that
\begin{align*}
g(x_0)=x_0\quad\forall g\in G.
\end{align*}

\begin{rmk}\label{Rmk_ubdd_orb}
In this context, having a global fixed point is equivalent to the existence of a bounded orbit, which can be seen from the identity
\begin{align*}
f\left(\sup_{g\in G}g(x)\right)=\sup_{g\in G}f(g(x))=\sup_{g\in G}g(x)\quad\forall f\in G.
\end{align*}
\end{rmk}

Now we can state our main result.

\begin{thm}\label{Thm_T_biLip}
Let $G$ be a finitely generated subgroup of $\operatorname{BiLip}_+^{\mathrm{bd}}(\R)$, and let $\Gamma$ be a subgroup of $G$ without global fixed points. If $(G,\Gamma)$ satisfies Property (T), then, for every finite generating set $S\subset G$,
\begin{align*}
\max_{g\in S}\operatorname{BiLip}(g)\geq \Phi\left(\kappa(G,\Gamma,S)\right),
\end{align*}
where $\Phi:[0,\sqrt{2})\to[1,\infty)$ is defined as
\begin{align}\label{def_Phi}
\Phi(t)=\max\left\{e^{2t},4(2-t^2)^{-2}\right\}\quad\forall t\in[0,\sqrt{2}).
\end{align}
\end{thm}

\begin{rmk}
The function $\Phi$ in Theorem \ref{Thm_T_biLip} is strictly increasing and surjective. We include the graphs of the two functions involved in its definition; see Figure \ref{Fig_graphs}.
\end{rmk}

\begin{figure}[h]
\begin{tikzpicture}
\begin{axis}[
    axis lines = left,
    ymin=0,
    xtick={0,0.2,0.4,0.6,0.8,1,1.2},
    ytick={0,2,4,6,8,10,12,14,16},
    ymajorgrids=true,
    xmajorgrids=true,
    grid style=dashed,
    legend pos=north west,
]
\addplot [
    domain=0:1.23, 
    samples=100, 
    color=green,
]
{exp(2*x)};
\addlegendentry{\(e^{2t}\)}
\addplot [
    domain=0:1.23, 
    samples=100, 
    color=blue,
    ]
    {4/(2-x^2)^2};
\addlegendentry{\(4(2-t^2)^{-2}\)}

\end{axis}
\end{tikzpicture}
\caption{Graphs of the functions involved in the definition of $\Phi$.}
\label{Fig_graphs}
\end{figure}

Theorem \ref{Thm_T_biLip} can be reinterpreted in the following way. Let $(G,\Gamma)$ be a pair with Property (T), and let $S\subset G$ be a finite generating set. Then $G$ cannot act faithfully on $\R$ by homeomorphisms in $\operatorname{BiLip}_+^{\mathrm{bd}}(\R)$ in such a way that $\Gamma$ does not have global fixed points and 
\begin{align*}
\operatorname{BiLip}(g) < \Phi\left(\kappa(G,\Gamma,S)\right)
\end{align*}
for all $g\in S$.

The main ingredient in the proof of Theorem \ref{Thm_T_biLip} is the Koopman representation on $L^p(\R)$, associated to the action $G\curvearrowright\R$, together with a version of Property (T) for $L^p$ spaces, as studied in \cite{FisMar} and \cite{BFGM}. The two functions involved in the definition of $\Phi$ are obtained from the cases $p=2$ and $p\to\infty$.

Next, we apply Theorem \ref{Thm_T_biLip} to the pair $(\F_2\ltimes\Z^2,\Z^2)$. Recall that the free group on two generators $\F_2$ can be realised as a normal, finite index subgroup of $\operatorname{SL}_2(\Z)$, with generators
\begin{align*}
\left(\begin{array}{cc}
1 & 2\\ 0 & 1
\end{array}\right), \qquad
\left(\begin{array}{cc}
1 & 0\\ 2 & 1
\end{array}\right);
\end{align*}
see e.g. \cite[Exercise 6.20]{CecCoo}. This representation of $\F_2$ is often referred to as the Sanov subgroup of $\operatorname{SL}_2(\Z)$. We can thus consider the semidirect product $\F_2\ltimes\Z^2$ for the natural action of $\operatorname{SL}_2(\Z)$ on $\Z^2$. In other words, an element of $\F_2\ltimes\Z^2$ is a pair $(g,n)$, where $g$ belongs to the Sanov subgroup and $n\in\Z^2$. The product in $\F_2\ltimes\Z^2$ is given by
\begin{align*}
(g,n)(h,m)=(gh,gm+n)\quad\forall g,h\in\F_2\ \ \forall n,m\in\Z^2.
\end{align*}
The pair $(\operatorname{SL}_2(\Z)\ltimes\Z^2,\Z^2)$ has Property (T); see e.g. \cite[Theorem 2.1]{Sha}. The same holds for $(\F_2\ltimes\Z^2,\Z^2)$ because $\F_2\ltimes\Z^2$ is a finite-index subgroup of $\operatorname{SL}_2(\Z)\ltimes\Z^2$. Moreover, the group $\F_2\ltimes\Z^2$ acts faithfully on $\R$ by orientation-preserving homeomorphisms; see \cite[\S 2]{Nav2} and the references therein. Hence we can apply Theorem \ref{Thm_T_biLip} to the pair $(\F_2\ltimes\Z^2,\Z^2)$ in order to obtain restrictions for the kinds of actions that $\F_2\ltimes\Z^2$ can have on $\R$. Let us consider the following generating subset of $\F_2\ltimes\Z^2$:
\begin{align}\label{gen_set_intro}
S=\left\{R, T, e_1, e_2\right\},
\end{align}
where
\begin{align*}
R&=\left(\begin{array}{cc}
1 &  2\\ 0 & 1
\end{array}\right), &
T&=\left(\begin{array}{cc}
1 & 0\\ 2 & 1
\end{array}\right),
\end{align*}
\begin{align*}
e_1&=\left(\begin{array}{c}
1 \\ 0
\end{array}\right), &
e_2&=\left(\begin{array}{c}
0 \\ 1
\end{array}\right).
\end{align*}
Here $R, T$ are elements of $\F_2$, and $e_1, e_2$ are elements of $\Z^2$, which we view as elements of $\F_2\ltimes\Z^2$ for the natural inclusions $\F_2, \Z^2\subset \F_2\ltimes\Z^2$.

In \cite[\S 2]{Sha}, Shalom proved a quantitative version of Property (T) for the pair $(\operatorname{SL}_2(\Z)\ltimes\Z^2,\Z^2)$; see also \cite[\S 4]{Kas}. By adapting his argument to the pair $(\F_2\ltimes\Z^2,\Z^2)$, and combining it with Theorem \ref{Thm_T_biLip}, we prove the following.

\begin{thm}\label{Thm_F2xZ2}
Let $\sigma:\F_2\ltimes\Z^2\to\operatorname{BiLip}_+^{\mathrm{bd}}(\R)$ be an injective homomorphism such that $\sigma(\Z^2)$ does not have global fixed points. Let $S$ be the generating set \eqref{gen_set_intro}. Then
\begin{align*}
\max_{g\in S}\operatorname{BiLip}(\sigma(g))\geq \exp\left(\frac{\sqrt{26}-4}{5}\right)\qquad(\approx 1.24).
\end{align*}
\end{thm}

As mentioned above, every finitely generated, orderable group can be realised as a subgroup of $\operatorname{BiLip}_+^{\mathrm{bd}}(\R)$ without global fixed points. We say that a group $G$ is (left) orderable if there is a total order $\prec$ on $G$ such that
\begin{align*}
g\prec f\quad\implies\quad hg\prec hf
\end{align*}
for all $f,g,h\in G$. We refer the reader to \cite{DeNaRi} for a thorough account on orderable groups. The main motivation behind this paper is the following open problem; see \cite[Remark 3.5.21]{DeNaRi}, and \cite[Question 3]{Nav} for a related question.

\begin{ques}
Does there exist an orderable group satisfying Property (T)?
\end{ques}

Many examples of groups with Property (T) have been shown not to be orderable; see \cite{DerHur, Orl, Orl2, Wit, Wit2}. But this question remains open in general. We do not aim to answer it here, but as a consequence of Theorem \ref{Thm_T_biLip}, we can obtain restrictions for the values of the Kazhdan constants of such groups, if they exist.

\begin{cor}\label{Cor_ord_Kazh}
Let $G$ be a finitely generated, orderable group, and let $\Gamma$ be a subgroup of $G$ such that $(G,\Gamma)$ has Property (T). Then, for every finite, symmetric generating subset $S\subset G$,
\begin{align*}
\kappa(G,\Gamma,S)\leq\Phi^{-1}(|S|),
\end{align*}
where $\Phi$ is defined as in \eqref{def_Phi}.
\end{cor}

\begin{rmk}
A consequence of Corollary \ref{Cor_ord_Kazh} is the following: if there exists an orderable group $G$ with Property (T), then
\begin{align*}
\kappa(G,S)\leq\Phi^{-1}(|S|),
\end{align*}
for every finite, symmetric generating subset $S\subset G$.
\end{rmk}

Observe that $\Phi^{-1}:[1,\infty)\to[0,\sqrt{2})$ is given by
\begin{align*}
\Phi^{-1}(t)=\min\left\{\tfrac{1}{2}\log(t),\sqrt{2}\left(1-t^{-1/2}\right)^{1/2}\right\}\quad \forall t\in[1,\infty).
\end{align*}
For clarity, we include below the graphs of the two functions involved in its definition.

\begin{figure}[h]
\begin{tikzpicture}
\begin{axis}[
    axis lines = left,
    xmin=0,
    xtick={0,2,4,6,8,10,12,14,16,18,20,22},
    ytick={0,0.25,0.5,0.75,1,1.25,1.5},
    xmajorgrids=true,
    ymajorgrids=true,
    grid style=dashed,
    legend pos=south east,
]
\addplot [
    domain=1:23, 
    samples=100, 
    color=green,
]
{ln(x)/2};
\addlegendentry{\(\frac{1}{2}\log(t)\)}
\addplot [
    domain=1:23, 
    samples=100, 
    color=blue,
    ]
    {sqrt(2-2/sqrt(x))};
\addlegendentry{\(\sqrt{2}(1-t^{-1/2})^{1/2}\)}

\end{axis}
\end{tikzpicture}
\caption{Graphs of the functions involved in the definition of $\Phi^{-1}$.}
\end{figure}

\subsection*{Organisation of the paper}
We begin by discussing ultralimits and the proof of Proposition \ref{Prop_inf_ab_q} in Section \ref{S_limits}. In Section \ref{S_rep_Lp}, we prove some preliminary results regarding the Mazur map and the Koopman representation on $L^p(\R)$. With all this, we prove Theorem \ref{Thm_T_biLip} in Section \ref{S_main_res}. In Section \ref{S_sdp}, we focus on the semidirect product $\F_2\ltimes\Z^2$, and we prove Theorem \ref{Thm_F2xZ2}. Finally, in Section \ref{S_ord}, we prove Corollary \ref{Cor_ord_Kazh}.

\subsection*{Acknowledgements} I am grateful to Andr\'es Navas for many interesting discussions, and for his very valuable comments and suggestions.

\section{Limits of actions with arbitrarily small Lipschitz constants}\label{S_limits}

In this section, we prove Proposition \ref{Prop_inf_ab_q}. The main idea of the proof is to construct an action by translations as a limit of actions with Lipschitz constants tending to $1$. This will be achieved through the use of ultralimits. For this purpose, we need to consider actions with a uniform control on the displacement.

\begin{lem}\label{Lem_conj_homoth}
Let $g\in\operatorname{BiLip}_+^{\mathrm{bd}}(\R)$. For $\alpha\in\R\setminus\{0\}$, let $\varphi_\alpha:\R\to\R$ denote the multiplication by $\alpha$. Then
\begin{align*}
\operatorname{BiLip}(\varphi_\alpha^{-1}\circ g\circ\varphi_\alpha)=\operatorname{BiLip}(g),
\end{align*}
and
\begin{align*}
\operatorname{disp}\left(\varphi_\alpha^{-1}\circ g\circ\varphi_\alpha\right) = \alpha^{-1}\operatorname{disp}(g).
\end{align*}
\end{lem}
\begin{proof}
For almost every $x\in\R$,
\begin{align*}
D(\varphi_\alpha^{-1}\circ g\circ\varphi_\alpha)(x) &=Dg(\alpha x),\\
D(\varphi_\alpha^{-1}\circ g^{-1}\circ\varphi_\alpha)(x) &=Dg^{-1}(\alpha x).
\end{align*}
This shows that $\operatorname{BiLip}(\varphi_\alpha^{-1}\circ g\circ\varphi_\alpha)=\operatorname{BiLip}(g)$; see Remark \ref{Rmk_Lip_D}. On the other hand,
\begin{align*}
\operatorname{disp}\left(\varphi_\alpha^{-1}\circ g\circ\varphi_\alpha\right)&=
\sup_{x\in\R}|\varphi_\alpha^{-1}\circ g\circ\varphi_\alpha(x)-x|\\
&= \sup_{x\in\R}|\alpha^{-1}g(\alpha x)-\alpha^{-1}\alpha x|\\
&= \alpha^{-1}\sup_{x\in\R}|g(x)-x|\\
&=\alpha^{-1}\operatorname{disp}(g).\qedhere
\end{align*}
\end{proof}

This result says that any $G<\operatorname{BiLip}_+^{\mathrm{bd}}(\R)$ can be conjugated into another subgroup of $\operatorname{BiLip}_+^{\mathrm{bd}}(\R)$ with the same Lipschitz constants, but for which the displacements are rescaled by a fixed amount.

Now we can prove Proposition \ref{Prop_inf_ab_q}. For details on ultrafilters and ultralimits, we refer the reader to \cite{Gol}; see also \cite[Appendix A.4]{Pis}.

\begin{proof}[Proof of Proposition \ref{Prop_inf_ab_q}]
We have a sequence of injective homomorphisms $\theta_n:G\to\operatorname{BiLip}_+^{\mathrm{bd}}(\R)$ such that
\begin{align*}
\lim_{n\to\infty}\max_{g\in S}\operatorname{BiLip}(\theta_n(g))=1,
\end{align*}
where $G=\langle S\rangle$. Conjugating the action $\theta_n$ by a homothecy, by Lemma \ref{Lem_conj_homoth}, we may assume that
\begin{align*}
\max_{g\in S}\operatorname{disp}(\theta_n(g)) = 1
\end{align*}
for all $n\in\N$. In particular, since $S$ generates $G$, for all $x\in\R$ and $g\in G$, the sequence $(\theta_n(g)(x))_{n\in\N}$ is bounded. Moreover, replacing $(\theta_n)$ by a subsequence, we can find $g_0\in S$ such that
\begin{align}\label{theta_n_g0}
\operatorname{disp}(\theta_n(g_0)) = \max_{g\in S}\operatorname{disp}(\theta_n(g)) = 1
\end{align}
for all $n\in\N$. Fix a non-principal ultrafilter $\mathfrak{U}$ on $\N$, and define
\begin{align}\label{limit_act}
g\cdot x=\lim_{\mathfrak{U}}\theta_n(g)(x)\quad\forall g\in G\ \ \forall x\in\R.
\end{align}
We claim that this defines an action by translations on $\R$. Indeed, let $e$ denote the identity element of $G$. Then
\begin{align*}
e\cdot x = \lim_{\mathfrak{U}}\theta_n(e)(x)= x
\end{align*}
for all $x\in\R$. Moreover, for every $f,g\in G$,
\begin{align*}
(fg)\cdot x - f\cdot g\cdot x &= \lim_{\mathfrak{U}}\theta_n(fg)(x)-\theta_n(f)(g\cdot x)\\
&= \lim_{\mathfrak{U}}\theta_n(f)(\theta_n(g)(x)-g\cdot x).
\end{align*}
By hypothesis, there is a constant $C_f\geq 1$ such that $\operatorname{BiLip}(\theta_n(f))\leq C_f$ for all $n\in\N$. Thus
\begin{align*}
|(fg)\cdot x - f\cdot g\cdot x| \leq C_f\lim_{\mathfrak{U}}|\theta_n(g)(x)-g\cdot x| = 0.
\end{align*}
This shows that \eqref{limit_act} defines an action of $G$ on $\R$. Moreover, for all $g\in G$ and $x,y\in\R$,
\begin{align*}
|g\cdot x - g\cdot y| \leq \lim_{\mathfrak{U}}\operatorname{BiLip}(\theta_n(g))|x-y|= |x-y|.
\end{align*}
Applying the same reasoning to $g^{-1}$, we obtain
\begin{align*}
|g\cdot x - g\cdot y| = |x-y|.
\end{align*}
Therefore \eqref{limit_act} defines an action by translations. Finally, the action is not trivial because
\begin{align*}
|g_0\cdot x-x|=\lim_{\mathfrak{U}}|\theta_n(g_0)(x)-x|=1,
\end{align*}
where $g_0$ is as in \eqref{theta_n_g0}. The map $g\mapsto g\cdot 0$ is a group homomorphism from $G$ to $\R$ with infinite image.
\end{proof}

\section{Property (T) and representations on $L^p$}\label{S_rep_Lp}

\subsection{Orthogonal representations}
Let $G$ be a group, and let $E$ be a real Banach space. The orthogonal group $\mathbf{O}(E)$ is the group of linear, surjective isometries of $E$. A group homomorphism $\pi:G\to\mathbf{O}(E)$ is called an orthogonal (or isometric) representation of $G$ on $E$. Let $\Gamma$ be a subgroup of $G$. We say that $\xi\in E$ is a $\Gamma$-invariant vector for the representation $\pi$ if
\begin{align*}
\pi(g)\xi=\xi\quad\forall g\in \Gamma.
\end{align*}

If $\mathcal{H}$ is a Hilbert space over $\C$, the group of linear, surjective isometries of $\mathcal{H}$ is called the unitary group of $\mathcal{H}$, and it is denoted by $\mathbf{U}(\mathcal{H})$. A group homomorphism $\pi:G\to\mathbf{U}(\mathcal{H})$ is called a unitary representation of $G$. As mentioned in the introduction, Property (T) is defined in terms of unitary representations; however, the proof of Theorem \ref{Thm_T_biLip} deals mainly with representations on real spaces. The following fact is a consequence of the complexification of orthogonal representations; see \cite[Remark A.7.2]{BedlHVa} and \cite[Remark 2.12.1]{BedlHVa}. We record it here for completeness.

\begin{lem}\label{Lem_Kazh_real}
Let $(G,\Gamma)$ be a pair with Property (T), and let $\pi:G\to\mathbf{O}(\mathcal{H})$ be an orthogonal representation on a real Hilbert space $\mathcal{H}$. Assume that $\pi$ does not have nontrivial $\Gamma$-invariant vectors. Then, for every finite generating set $S\subset G$,
\begin{align*}
\max_{g\in S}\|\pi(g)\xi-\xi\|\geq\kappa(G,\Gamma,S)\|\xi\|\quad\forall\xi\in\mathcal{H},
\end{align*}
where $\kappa(G,\Gamma,S)$ is the Kazhdan constant associated to $(G,\Gamma,S)$, as defined in \eqref{Kazh_const}.
\end{lem}

\subsection{The Mazur map and representations on $L^p$}
In order to obtain the exponential bound in Theorem \ref{Thm_T_biLip}, we need to consider representations on $L^p(\R)$ for $p\geq 2$. The Mazur map allows one to relate such representations for different values of $p$. Let $(X,\mu)$ be a measure space, and let $1\leq p,q<\infty$. The Mazur map $M_{q,p}:L^q(X,\mu)\to L^p(X,\mu)$ is defined by
\begin{align}\label{def_Mazur}
M_{q,p}(\xi)=\operatorname{sign}(\xi)|\xi|^\frac{q}{p}\quad\forall\xi\in L^q(X,\mu).
\end{align}
The proof the following result can be found in \cite[\S 9.1]{BenLin}; see also \cite[\S 3.7.1]{Now}.

\begin{thm}[Mazur]\label{Thm_Maz_map}
Let $(X,\mu)$ be a measure space, and let $1\leq q\leq p<\infty$. Then the Mazur map \eqref{def_Mazur} is a homeomorphism between the unit spheres of $L^q(X,\mu)$ and $L^p(X,\mu)$. More precisely, there is a constant $C>0$ (depending only on $\frac{q}{p}$) such that
\begin{align*}
\frac{q}{p}\|\xi-\eta\|_q\leq \|M_{q,p}(\xi)-M_{q,p}(\eta)\|_p\leq C\|\xi-\eta\|_q^{\frac{q}{p}},
\end{align*}
for all $\xi,\eta$ in the unit sphere of $L^q(X,\mu)$.
\end{thm}

Although the Mazur map is nonlinear, it conjugates orthogonal representations into orthogonal representations. This is a consequence of the Banach--Lamperti
theorem, which describes the isometries of $L^p$; see e.g. \cite[Theorem 2.16]{BFGM}. Moreover, it was shown in \cite[Theorem A]{BFGM} that a group with Property (T) satisfies an analogous property for representations on $L^p$ spaces. We will need a quantitative version of this fact.

\begin{prop}\label{Prop_p-Kazh}
Let $(X,\mu)$ be a $\sigma$-finite measure space and $p\in[2,\infty)$. Let $(G,\Gamma)$ be a pair with Property (T), and let $\pi:G\to\mathbf{O}(L^p(X,\mu))$ be an orthogonal representation without nontrivial $\Gamma$-invariant vectors. Then, for every finite generating subset $S\subset G$,
\begin{align*}
\max_{g\in S}\|\pi(g)\xi-\xi\|_p \geq\frac{2}{p}\kappa(G,\Gamma,S)\|\xi\|_p\quad\forall\xi\in L^p(X,\mu),
\end{align*}
where $\kappa(G,\Gamma,S)$ is the Kazhdan constant, as defined in \eqref{Kazh_const}.
\end{prop}
\begin{proof}
By \cite[Lemma 4.2]{BFGM}, we have an orthogonal representation $\tilde{\pi}:G\to\mathbf{O}(L^2(X,\mu))$ given by
\begin{align*}
\tilde{\pi}(g)=M_{p,2}\circ\pi(g)\circ M_{2,p}\quad\forall g\in G,
\end{align*}
where $M_{2,p}$ and $M_{p,2}$ are defined as in \eqref{def_Mazur}. Now let $\xi$ be a unit vector in $L^p(X,\mu)$ and $g\in S$. By Theorem \ref{Thm_Maz_map},
\begin{align*}
\|\pi(g)\xi-\xi\|_p &= \|M_{2,p}\circ\tilde{\pi}(g)\circ M_{p,2}(\xi)-M_{2,p}\circ M_{p,2}(\xi)\|_p\\
&\geq\frac{2}{p}\|\tilde{\pi}(g)M_{p,2}(\xi)-M_{p,2}(\xi)\|_2.
\end{align*}
A similar argument, using the other inequality in Theorem \ref{Thm_Maz_map}, shows that $\tilde{\pi}$ does not have nontrivial $\Gamma$-invariant vectors. Combining this with Lemma \ref{Lem_Kazh_real}, we get
\begin{align*}
\max_{g\in S}\|\pi(g)\xi-\xi\|_p \geq\frac{2}{p}\kappa(G,\Gamma,S)\|M_{p,2}(\xi)\|_2 = \frac{2}{p}\kappa(G,\Gamma,S).
\end{align*}
This holds for every unit vector in $L^p(X,\mu)$. By homogeneity, we obtain the desired inequality for every $\xi\in L^p(X,\mu)$.
\end{proof}

\subsection{The Koopman representation for subgroups of $\operatorname{BiLip}_+^{\mathrm{bd}}(\R)$}
Let $G$ be a group acting on a measure space $(X,\mu)$ by measure class preserving transformations. This means that, for every $g\in G$, the measures $\mu$ and $g_*\mu$ are absolutely continuous with respect to each other. Then, for every $p\in[1,\infty)$, we can define the Koopman representation $\pi:G\to\mathbf{O}(L^p(X,\mu))$ by
\begin{align*}
\pi(g)\xi(x)=\xi(g^{-1}(x))\left(\frac{d(g_*\mu)}{d\mu}(x)\right)^{\frac{1}{p}}\quad\forall g\in G\ \ \forall \xi\in L^p(X,\mu)\ \ \forall x\in X.
\end{align*}
If $G$ is a subgroup of $\operatorname{BiLip}_+^{\mathrm{bd}}(\R)$, then the action of $G$ on $\R$ preserves the measure class of the Lebesgue measure. In this case, the Koopman representation $\pi:G\to\mathbf{O}(L^p(\R))$ is given by
\begin{align}\label{Koop_Leb}
\pi(g)\xi(x)=\xi(g^{-1}(x))Dg^{-1}(x)^{\frac{1}{p}}\quad\forall g\in G\ \ \forall\xi\in L^p(\R)\ \ \forall x\in\R.
\end{align}
This representation will be the main ingredient in the proof of Theorem \ref{Thm_T_biLip}. First, we show that the existence of nontrivial invariant vectors implies the existence of global fixed points.

\begin{lem}\label{Lem_no_inv_v}
Let $\Gamma$ be a subgroup of $\operatorname{BiLip}_+^{\mathrm{bd}}(\R)$, and let $\pi:\Gamma\to\mathbf{O}(L^p(\R))$ be the associated Koopman representation for $p\in[1,\infty)$. If $\Gamma$ does not have global fixed points, then $\pi$ does not have non-trivial invariant vectors.
\end{lem}
\begin{proof}
Assume by contradiction that $\xi$ is a $\pi$-invariant vector with $\|\xi\|_p=1$. Take $M>0$ such that
\begin{align*}
\int_{-M}^M|\xi(x)|^p\, dx >\frac{1}{2}.
\end{align*}
By Remark \ref{Rmk_ubdd_orb}, the action $\Gamma\curvearrowright\R$ has unbounded orbits. Hence there is $g\in \Gamma$ such that
\begin{align*}
[-M,M]\cap[g(-M),g(M)]=\varnothing.
\end{align*}
Since $\pi(g)\xi=\xi$, we have
\begin{align*}
1 &= \|\xi\|_p^p\\
&\geq \int_{-M}^M|\xi(x)|^p\, dx + \int_{g(-M)}^{g(M)}|\pi(g)\xi(x)|^p\, dx\\
&= 2\int_{-M}^M|\xi(x)|^p\, dx\\
&> 1,
\end{align*}
which is a contradiction.
\end{proof}

\section{Proof of the main result}\label{S_main_res}
This section is devoted to the proof of Theorem \ref{Thm_T_biLip}. We will show that, under the hypotheses of Theorem \ref{Thm_T_biLip},
\begin{align}\label{kappa<Phi-1}
\kappa(G,\Gamma,S)\leq\Phi^{-1}\left(\max_{g\in S}\operatorname{BiLip}(g)\right),
\end{align}
where $\Phi^{-1}:[1,\infty)\to[0,\sqrt{2})$ is given by
\begin{align*}
\Phi^{-1}(t)=\min\left\{\tfrac{1}{2}\log(t),\sqrt{2}\left(1-t^{-1/2}\right)^{1/2}\right\}\quad \forall t\in[1,\infty).
\end{align*}

\subsection{An upper bound for $\kappa(G,\Gamma,S)$ for large values of $\operatorname{BiLip}(g)$}
We begin with the bound $\kappa(G,\Gamma,S)\leq\sqrt{2}\left(1-\operatorname{BiLip}(g)^{-1/2}\right)^{1/2}$ in \eqref{kappa<Phi-1}.

\begin{lem}\label{Lem_estimate1}
Let $G$ be a finitely generated subgroup of $\operatorname{BiLip}_+^{\mathrm{bd}}(\R)$, and let $\Gamma$ be a subgroup of $G$ without global fixed points. If the pair $(G,\Gamma)$ satisfies Property (T), then, for every finite generating set $S\subset G$,
\begin{align*}
\kappa(G,\Gamma,S)\leq \max_{g\in S}\sqrt{2}\left(1-\operatorname{BiLip}(g)^{-1/2}\right)^{1/2}.
\end{align*}
\end{lem}
\begin{proof}
Let $\pi:G\to\mathbf{O}(L^2(\R))$ be the Koopman representation, defined as in \eqref{Koop_Leb}. By Lemma \ref{Lem_no_inv_v}, this representation does not have nontrivial $\Gamma$-invariant vectors. Fix a finite generating set $S\subset G$, and define
\begin{align*}
M&=\max_{g\in S}\operatorname{disp}(g),\\
L&=\max_{g\in S}\operatorname{BiLip}(g).
\end{align*}
For every $n\geq 1$, let us define a unit vector $\xi_n\in L^2(\R)$ by
\begin{align*}
\xi_n=(2n)^{-\frac{1}{2}}\mathds{1}_{[-n,n]},
\end{align*}
where $\mathds{1}_{[-n,n]}$ denotes the indicator function of the interval $[-n,n]$. For every $g\in S$,
\begin{align*}
\|\pi(g)\xi_n-\xi_n\|^2 =2-2\langle\pi(g)\xi_n,\xi_n\rangle,
\end{align*}
and
\begin{align*}
\langle\pi(g)\xi_n,\xi_n\rangle
&= (2n)^{-1}\int_{-n}^n \mathds{1}_{[g(-n),g(n)]}(x)Dg^{-1}(x)^{\frac{1}{2}}\, dx\\
&\geq (2n)^{-1}L^{-\frac{1}{2}}\int_{-n}^n \mathds{1}_{[g(-n),g(n)]}(x)\, dx.
\end{align*}
For $n$ large enough, we have
\begin{align*}
\max\{-n,g(-n)\} \leq -n+M \leq n-M \leq \min\{n,g(n)\},
\end{align*}
which shows that
\begin{align*}
\langle\pi(g)\xi_n,\xi_n\rangle &\geq (2n)^{-1}L^{-\frac{1}{2}}(n-M-(-n+M))\\
&=\frac{n-M}{n}L^{-\frac{1}{2}}.
\end{align*}
Hence, since $\|\xi_n\|=1$, by Lemma \ref{Lem_Kazh_real},
\begin{align*}
\kappa(G,\Gamma,S)^2 &\leq \|\pi(g)\xi_n-\xi_n\|^2\\
&\leq 2-\frac{2(n-M)}{n}L^{-\frac{1}{2}}
\end{align*}
Taking the limit $n\to\infty$, we obtain
\begin{align*}
\kappa(G,\Gamma,S) & \leq \sqrt{2}\left(1-L^{-\frac{1}{2}}\right)^\frac{1}{2}\\
&= \max_{g\in S}\sqrt{2}\left(1-\operatorname{BiLip}(g)^{-1/2}\right)^{1/2}.\qedhere
\end{align*}
\end{proof}

\subsection{An upper bound for $\kappa(G,\Gamma,S)$ for small values of $\operatorname{BiLip}(g)$}

Now we deal with the logarithmic bound in \eqref{kappa<Phi-1}. For this purpose, we need to consider the Koopman representation on $L^p(\R)$. We will make use of the following estimate.

\begin{lem}\label{Lem_log(L)}
Let $L>1$, $x\in[L^{-1},L]$, and $p>\log(L)$. Then
\begin{align*}
|x^{\frac{1}{p}}-1|\leq\frac{\log(L)}{p-\log(L)}.
\end{align*}
\end{lem}
\begin{proof}
We will use the following classical inequalities for the exponential function:
\begin{align*}
1+t\leq e^t \leq 1+\frac{t}{1-t}\quad \forall t\in(-\infty,1).
\end{align*}
For $t=\frac{1}{p}\log(x)$, we get
\begin{align*}
\frac{1}{p}\log(x)\leq x^{\frac{1}{p}}-1\leq \frac{\log(x)}{p-\log(x)}.
\end{align*}
Assume first that $x\geq 1$. Then
\begin{align*}
|x^{\frac{1}{p}}-1|=x^{\frac{1}{p}}-1\leq \frac{\log(x)}{p-\log(x)}\leq\frac{\log(L)}{p-\log(L)}.
\end{align*}
On the other hand, if $x<1$, then
\begin{align*}
|x^{\frac{1}{p}}-1|=1-x^{\frac{1}{p}}\leq \frac{1}{p}\log(x^{-1})\leq\frac{\log(L)}{p}\leq\frac{\log(L)}{p-\log(L)}.&\qedhere
\end{align*}
\end{proof}

Now we are ready to obtain our second estimate for $\kappa(G,\Gamma,S)$.

\begin{lem}\label{Lem_estimate2}
Let $G$ be a finitely generated subgroup of $\operatorname{BiLip}_+^{\mathrm{bd}}(\R)$, and let $\Gamma$ be a subgroup of $G$ without global fixed points. If the pair $(G,\Gamma)$ satisfies Property (T), then, for every finite generating set $S\subset G$,
\begin{align*}
\kappa(G,\Gamma,S)\leq \frac{1}{2}\max_{g\in S}\log(\operatorname{BiLip}(g)).
\end{align*}
\end{lem}
\begin{proof}
Let $S\subset G$ be a finite generating set. As in the proof of Lemma \ref{Lem_estimate1}, let us define the following constants:
\begin{align*}
M&=\max_{g\in S}\operatorname{disp}(g),\\
L&=\max_{g\in S}\operatorname{BiLip}(g).
\end{align*}
Let $p>\log(L)$, and let $\pi:G\to\mathbf{O}(L^p(\R))$ be the Koopman representation:
\begin{align*}
\pi(g)\xi(x)=\xi(g^{-1}(x))Dg^{-1}(x)^{\frac{1}{p}}\quad\forall g\in G\ \ \forall\xi\in L^p(\R)\ \ \forall x\in\R.
\end{align*}
By Lemma \ref{Lem_no_inv_v}, $\pi$ does not have nontrivial $\Gamma$-invariant vectors. Hence, by Proposition \ref{Prop_p-Kazh}, we have
\begin{align}\label{s_gap_Lp}
\max_{g\in S}\|\pi(g)\xi-\xi\|_p\geq\frac{2}{p}\kappa(G,\Gamma,S)\|\xi\|_p\quad\forall\xi\in L^p(\R).
\end{align}
For every $n\geq 1$, let us define
\begin{align*}
\xi_n=(2n)^{-\frac{1}{p}}\mathds{1}_{[-n,n]}.
\end{align*}
Observe that $\|\xi_n\|_p=1$ for all $n\geq 1$. On the other hand, for every $g\in S$,
\begin{align*}
\|\pi(g)\xi_n-\xi_n\|_p^p=(2n)^{-1}\int_{-\infty}^\infty \left|\mathds{1}_{[g(-n),g(n)]}(x)Dg^{-1}(x)^{\frac{1}{p}}-\mathds{1}_{[-n,n]}(x)\right|^p\, dx.
\end{align*}
Now, for every $n$ large enough,
\begin{align*}
-n-M \leq \min\{-n,g(-n)\} \leq \max\{-n,g(-n)\} \leq -n+M,
\end{align*}
and
\begin{align*}
n-M \leq \min\{n,g(n)\} \leq \max\{n,g(n)\} \leq n+M.
\end{align*}
Hence
\begin{align*}
2n\|\pi(g)\xi_n-\xi_n\|_p^p &\leq \max\left\{\int_{-n-M}^{-n+M}Dg^{-1}(x)\, dx, \int_{-n-M}^{-n+M}1\, dx\right\}\\
&\quad + \int_{-n-M}^{n+M}\left|Dg^{-1}(x)^{\frac{1}{p}}-1\right|^p\, dx\\
&\quad + \max\left\{\int_{n-M}^{n+M}Dg^{-1}(x)\, dx, \int_{n-M}^{n+M}1\, dx\right\}\\
&\leq 4ML + \int_{-n-M}^{n+M}\left|Dg^{-1}(x)^{\frac{1}{p}}-1\right|^p\, dx.
\end{align*}
Thus, by Lemma \ref{Lem_log(L)}, for every $g\in S$ and every $n$ large enough,
\begin{align*}
\|\pi(g)\xi_n-\xi_n\|_p^p\leq \frac{4ML}{2n} + \frac{2(n+M)}{2n}\left(\frac{\log(L)}{p-\log(L)}\right)^p.
\end{align*}
Therefore
\begin{align*}
\limsup_{n\to\infty}\max_{g\in S}\|\pi(g)\xi_n-\xi_n\|_p\leq \frac{\log(L)}{p-\log(L)}.
\end{align*}
Combining this with \eqref{s_gap_Lp}, we get
\begin{align*}
\frac{2}{p}\kappa(G,\Gamma,S)\leq \frac{\log(L)}{p-\log(L)}.
\end{align*}
Equivalently,
\begin{align*}
\kappa(G,\Gamma,S)\leq \frac{\log(L)}{2\left(1-\frac{1}{p}\log(L)\right)}.
\end{align*}
Letting $p\to\infty$, this yields
\begin{align*}
\kappa(G,\Gamma,S)&\leq\frac{1}{2}\log(L)\\
&=\frac{1}{2}\log\left(\max_{g\in S}\operatorname{BiLip}(g)\right).\qedhere
\end{align*}
\end{proof}

We can now prove Theorem \ref{Thm_T_biLip}.

\begin{proof}[Proof of Theorem \ref{Thm_T_biLip}]
Let $S$ be a finite generating set of $G$, and let $\kappa(G,\Gamma,S)$ be the associated Kazhdan constant. By Lemmas \ref{Lem_estimate1} and \ref{Lem_estimate2},
\begin{align*}
\kappa(G,\Gamma,S)\leq\min\left\{\tfrac{1}{2}\log(L),\sqrt{2}\left(1-L^{-1/2}\right)^{1/2}\right\}=\Phi^{-1}(L),
\end{align*}
where
\begin{align*}
L=\max_{g\in S}\operatorname{BiLip}(g).
\end{align*}
Hence $L\geq\Phi(\kappa(G,\Gamma,S))$, where $\Phi$ is defined as in \eqref{def_Phi}.
\end{proof}

\section{The case of $\F_2\ltimes\Z^2$}\label{S_sdp}

This section is devoted to the proof of Theorem \ref{Thm_F2xZ2}. The main ingredient is an estimate for the Kazhdan constant of $(\F_2\ltimes\Z^2,\Z^2)$, which is essentially an adaptation of \cite[\S 2]{Sha} to our setting. We will denote by $\F_2$ the (free) subgroup of $\operatorname{SL}_2(\Z)$ generated by the matrices
\begin{align*}
\left(\begin{array}{cc}
1 & 2 \\ 0 & 1
\end{array}\right), &&
\left(\begin{array}{cc}
1 & 0 \\ 2 & 1
\end{array}\right).
\end{align*}
The semidirect product $\F_2\ltimes\Z^2$ is defined as the set $\F_2\times\Z^2$, endowed with the product
\begin{align*}
(g,n)(h,m)=(gh,gm+n)\quad\forall g,h\in\F_2\ \ \forall n,m\in\Z^2,
\end{align*}
where the action of $\F_2$ on $\Z^2$ is given by matrix multiplication. We identify $\F_2$ and $\Z^2$ with the subgroups $\F_2\times\{0\}$ and $\{\mathrm{Id}\}\times\Z^2$ respectively. Under this identification, the action of $\F_2$ on $\Z^2$ becomes the action by conjugation in $\F_2\ltimes\Z^2$. In particular, if $\pi:\F_2\ltimes\Z^2\to\mathbf{U}(\mathcal{H})$ is a unitary representation,
\begin{align*}
\pi(\mathrm{Id},gn)=\pi(g,0)\pi(\mathrm{Id},n)\pi(g,0)^{-1}\quad\forall g\in\F_2\ \ \forall n\in\Z^2,
\end{align*}
which we will write in the following simplified form:
\begin{align}\label{pi(gn)}
\pi(gn)=\pi(g)\pi(n)\pi(g)^{-1}\quad\forall g\in\F_2\ \ \forall n\in\Z^2.
\end{align}

\subsection{Spectral measures}
We review now the construction of the projection-valued measure associated to a unitary representation of $\Z^2$. This may be viewed as a particular case of the spectral theorem for two commuting normal operators; we refer the reader to \cite[Chapter VII]{ReeSim} and \cite[Chapter 7]{Wei} for details.

Let
\begin{align*}
\T^2=\T\times\T=\left\{(z_1,z_2)\in\C^2\ \mid\ |z_1|=|z_2|=1\right\},
\end{align*}
and let $C(\T^2)$ denote the algebra of continuous functions from $\T^2$ to $\C$, which can be identified with the full group $\mathbf{C}^*$-algebra $\mathrm{C}^*(\Z^2)$ via the Fourier transform; see e.g. \cite[\S 7.1]{Fol}. Let $\pi:\Z^2\to\mathbf{U}(\mathcal{H})$ be a unitary representation, and let $\mathbf{B}(\mathcal{H})$ denote the algebra of bounded operators on $\mathcal{H}$. Then there is a unital $\ast$-homomorphism $\Psi:C(\T^2)\to\mathbf{B}(\mathcal{H})$ given by
\begin{align*}
\Psi(f)=\sum_{n\in\Z^2}\hat{f}(n)\pi(n)\quad\forall f\in C(\T^2),
\end{align*}
where
\begin{align*}
\hat{f}(n)=\int_{\T}\int_{\T}f(z_1,z_2)z_1^{-n_1}z_2^{-n_2}\,dz_1\,dz_2\quad\forall n=(n_1,n_2)\in\Z^2.
\end{align*}
Here $dz_i$ denotes the integration with respect to the normalised Lebesgue measure on $\T$. For every $\xi\in\mathcal{H}$, we can define a Borel measure $\mu_\xi$ on $\T^2$ by
\begin{align*}
\int_{\T^2}f(z)\,d\mu_\xi(z) = \langle\Psi(f)\xi,\xi\rangle\quad\forall f\in C(\T^2).
\end{align*}
This allows us to extend $\Psi$ to the algebra of bounded Borel functions $\mathcal{B}_b(\T^2)$ in the following way. For every $f\in \mathcal{B}_b(\T^2)$, we define $\Psi(f)\in\mathbf{B}(\mathcal{H})$ as the unique operator satisfying
\begin{align*}
\langle\Psi(f)\xi,\xi\rangle = \int_{\T^2}f(z)\,d\mu_\xi(z)
\end{align*}
for every $\xi\in\mathbf{B}(\mathcal{H})$. More precisely, we can define $\Psi(f)\xi$ implicitly via the following polarisation-type formula:
\begin{align*}
\langle\Psi(f)\xi,\eta\rangle=\frac{1}{4}\sum_{k=0}^3\ii^k\int_{\T^2}f(z)\,d\mu_{(\xi+\ii^k\eta)}(z)\quad\forall\eta\in\mathcal{H}.
\end{align*}
In particular, for every Borel subset $B\subseteq\T^2$, there is a projection $\mathrm{P}(B)$ on $\mathcal{H}$ given by
\begin{align}\label{P(B)=Psi}
\mathrm{P}(B)=\Psi(\mathds{1}_B).
\end{align}
We say that $\mathrm{P}$ is a projection-valued measure on $\T^2$. The following proposition seems to be well known. We include its proof here for completeness.

\begin{prop}\label{Prop_pvm}
Let $\pi:\F_2\ltimes\Z^2\to\mathbf{U}(\mathcal{H})$ be a unitary representation, and let $\mathrm{P}$ be the projection-valued measure associated to the restriction of $\pi$ to $\Z^2$. Then $\mathrm{P}$ satisfies the following properties:
\begin{itemize}
\item[a)] The operator $\mathrm{P}(\{(1,1)\})$ is the orthogonal projection onto the subspace of $\Z^2$-invariant vectors of $\mathcal{H}$.
\item[b)] For every Borel subset $B\subseteq\T^2$ and every $g\in\F_2$,
\begin{align*}
\mathrm{P}\left(\left(g^\intercal\right)^{-1}B\right)=\pi(g)\mathrm{P}(B)\pi(g)^{-1},
\end{align*}
where $g^\intercal$ denotes the transpose of $g$, and the action of $\F_2$ on $\T^2$ is the matrix multiplication on the quotient $\R^2/\Z^2$.
\end{itemize}
\end{prop}
\begin{proof}
Let $\mathcal{H}_0$ denote the subspace of $\Z^2$-invariant vectors of $\mathcal{H}$. For every $n\in\Z^2$, let $e_n$ be the function
\begin{align}\label{e_n}
e_n(z)=z_1^{n_1}z_2^{n_2}\quad\forall z=(z_1,z_2)\in\T^2,
\end{align}
where $n=(n_1,n_2)$. Then, for all $\xi\in\mathcal{H}$ and $n\in\Z^2$,
\begin{align*}
\pi(n)\mathrm{P}\left(\{(1,1)\}\right)\xi=\Psi\left(e_n\mathds{1}_{\{(1,1)\}}\right)\xi=\Psi\left(\mathds{1}_{\{(1,1)\}}\right)\xi=\mathrm{P}\left(\{(1,1)\}\right)\xi,
\end{align*}
which shows that $\mathrm{P}(\{(1,1)\})\xi$ belongs to $\mathcal{H}_0$. Hence $\mathrm{P}(\{(1,1)\})$ is the orthogonal projection onto a subspace of $\mathcal{H}_0$. In order to see that this subspace is actually $\mathcal{H}_0$, let us consider the Fej\'er kernel:
\begin{align*}
F_N(w)=\sum_{n=-N}^{N}\left(1-\frac{|n|}{N+1}\right)w^n\quad\forall w\in\T.
\end{align*}
Then the sequence of functions $f_N\in C(\T^2)$ given by
\begin{align*}
f_N(z_1,z_2)=\frac{1}{N^2}F_N(z_1)F_N(z_2)\quad\forall z_1,z_2\in\T,
\end{align*}
satisfies
\begin{align*}
\sup_{N\geq 0}\sup_{z\in\T^2}|f_N(z)|\leq 4,
\end{align*}
and
\begin{align*}
f_N(z)\xrightarrow[N\to\infty]{}\begin{cases}
1, & z=(1,1),\\
0, & z\neq (1,1),
\end{cases}
\end{align*}
for all $z\in\T^2$; see e.g. \cite[Proposition 3.1.7]{Gra}. Therefore, by the dominated convergence theorem, for every $\xi\in\mathcal{H}$,
\begin{align*}
\mathrm{P}(\{(1,1)\})\xi =\lim_{N\to\infty}\Psi(f_N)\xi.
\end{align*}
Now assume that $\xi$ belongs to $\mathcal{H}_0$. Then
\begin{align*}
\mathrm{P}(\{(1,1)\})\xi 
&=\lim_{N\to\infty}\frac{1}{N^2}\sum_{n_1,n_2=-N}^{N}\left(1-\frac{|n_1|}{N+1}\right)\left(1-\frac{|n_2|}{N+1}\right)\pi(n_1,n_2)\xi\\
&=\lim_{N\to\infty}\frac{1}{N^2}\left(2N+1-\frac{2}{N+1}\frac{N(N+1)}{2}\right)^2\xi\\
&=\xi.
\end{align*}
We conclude that $\mathrm{P}(\{(1,1)\})$ is the orthogonal projection onto $\mathcal{H}_0$. This proves (a).\\
In order to prove (b), let us consider first the case of a continuous function $f\in C(\T^2)$. For $g\in\F_2$ and $z\in\T^2$, let us write
\begin{align*}
g\cdot f(z)=f(g^{-1}z),
\end{align*}
where the action of $\F_2$ on $\T^2$ comes from the natural action of $\operatorname{SL}_2(\Z)$ on $\R^2$. More explicitly, if
\begin{align*}
g=\left(\begin{array}{cc}
a & b \\ c & d
\end{array}\right),
\end{align*}
then
\begin{align}\label{g(z1,z2)}
g\left(z_1,z_2\right)=\left(z_1^az_2^{b},z_1^{c}z_2^d\right).
\end{align}
We have
\begin{align*}
\Psi(g\cdot f)=\sum_{n\in\Z^2}\widehat{g\cdot f}(n)\pi(n),
\end{align*}
where
\begin{align*}
\widehat{g\cdot f}(n)=\int_{\T^2}f(g^{-1}z)e_{-n}(z)\, dz,
\end{align*}
and $e_{-n}$ is defined as in \eqref{e_n}. Since the determinant of $g$ is 1, we obtain
\begin{align*}
\widehat{g\cdot f}(n)=\int_{\T^2}f(z)e_{-n}(gz)\, dz.
\end{align*}
On the other hand, from \eqref{g(z1,z2)}, one sees that
\begin{align*}
e_{-n}(gz)=e_{-g^\intercal n}(z).
\end{align*}
Therefore
\begin{align*}
\Psi(g\cdot f)=\sum_{n\in\Z^2}\hat{f}(g^\intercal n)\pi(n)=\sum_{n\in\Z^2}\hat{f}(n)\pi\left(\left(g^\intercal\right)^{-1}n\right).
\end{align*}
Replacing $g$ by $\left(g^\intercal\right)^{-1}$, and using the identity \eqref{pi(gn)},
\begin{align*}
\Psi\left(\left(g^\intercal\right)^{-1}\cdot f\right)=\sum_{n\in\Z^2}\hat{f}(n)\pi(g)\pi(n)\pi(g)^{-1}=\pi(g)\Psi(f)\pi(g)^{-1}.
\end{align*}
Now let $B\subseteq\T^2$ be a Borel subset. We can approximate $\mathds{1}_{B}$ by a sequence of continuous functions, and use the identity above to obtain
\begin{align*}
\mathrm{P}\left(\left(g^\intercal\right)^{-1}B\right)=\Psi\left((g^\intercal)^{-1}\cdot\mathds{1}_{B}\right)
=\pi(g)\Psi\left(\mathds{1}_{B}\right)\pi(g)^{-1}=\pi(g)\mathrm{P}(B)\pi(g)^{-1}.&\qedhere
\end{align*}
\end{proof}

\subsection{Shalom's argument revisited}
Shalom proved in \cite[Theorem 2.1]{Sha} that the Kazhdan constant of $(\operatorname{SL}_2(\Z)\ltimes\Z^2,\Z^2)$, associated to a generating set consisting of elementary matrices and the canonical basis of $\Z^2$, is at least $\frac{1}{10}$. We explain here how his argument can be adapted to our setting. Let $R^{\pm},T^{\pm}\in\F_2$ be defined as
\begin{align}\label{R+-T+-}
R^{\pm}&=\left(\begin{array}{cc}
1 & \pm 2\\ 0 & 1
\end{array}\right), &
T^{\pm}&=\left(\begin{array}{cc}
1 & 0\\ \pm 2 & 1
\end{array}\right),
\end{align}
and let $\pm e_1, \pm e_2\in\Z^2$ be
\begin{align*}
\pm e_1&=\left(\begin{array}{c}
\pm 1 \\ 0
\end{array}\right), &
\pm e_2&=\left(\begin{array}{c}
0 \\ \pm 1
\end{array}\right).
\end{align*}
We consider the symmetric generating subset of $\F_2\ltimes\Z^2$ given by
\begin{align}\label{gen_sset_sdp}
S=\left\{R^{\pm}, T^{\pm}, \pm e_1, \pm e_2\right\},
\end{align}
where $\F_2$ and $\Z^2$ are naturally embedded in $\F_2\ltimes\Z^2$.

The first step consists in considering a unitary representation of $\F_2\ltimes\Z^2$, and translating the almost invariance of a vector into the almost invariance of a certain probability measure on $\T^2$.

\begin{lem}\label{Lem_exist_pm}
Let $\pi:\F_2\ltimes\Z^2\to\mathbf{U}(\mathcal{H})$ be a unitary representation without nontrivial $\Z^2$-invariant vectors. Let $\varepsilon\in\big(0,\frac{1}{\sqrt{2}}\big)$, and assume that there is a unit vector $\xi\in\mathcal{H}$ such that
\begin{align*}
\max_{g\in S}\|\pi(g)\xi-\xi\|<\varepsilon,
\end{align*}
where $S$ is the subset defined in \eqref{gen_sset_sdp}. Then there is a Borel probability measure $\nu$ on ${\big(-\tfrac{1}{2},\tfrac{1}{2}\big]^2}$, supported on $\big(-\frac{1}{6},\frac{1}{6}\big)^2\setminus\{(0,0)\}$, such that, for every Borel set $B\subseteq\big(-\tfrac{1}{2},\tfrac{1}{2}\big]^2$ and every $g\in\left\{R^{\pm}, T^{\pm}\right\}$,
\begin{align}\label{nu_almost_inv}
\left|\nu(gB)-\nu(B)\right|<\frac{2(\varepsilon+\varepsilon^2)}{1-2\varepsilon^2},
\end{align}
where we identify $\big(-\tfrac{1}{2},\tfrac{1}{2}\big]^2\cong \R^2/\big(\Z^2+\big(\frac{1}{2},\frac{1}{2}\big)\big)\cong\T^2$.
\end{lem}
\begin{proof}
Let $\mathrm{P}$ be the projection-valued measure associated to the restriction of $\pi$ to $\Z^2$, defined as in \eqref{P(B)=Psi}. Let $\mu_\xi$ be the Borel probability measure on $\T^2$ given by
\begin{align*}
\mu_\xi(B)=\langle\mathrm{P}(B)\xi,\xi\rangle,
\end{align*}
which we view as a measure on $\big(-\tfrac{1}{2},\tfrac{1}{2}\big]^2$ under the identification described above. Since $\pi$ does not have $\Z^2$-invariant vectors, by Proposition \ref{Prop_pvm}.(a),
\begin{align}\label{mu_xi(0)}
\mu_\xi(\{(0,0)\})=0.
\end{align}
Let $\Omega=\big(-\frac{1}{6},\frac{1}{6}\big)^2$. We claim that $\mu_\xi(\Omega)\geq 1-2\varepsilon^2$. Indeed, observe that
\begin{align*}
\int\limits_{\left(-\frac{1}{2},\frac{1}{2}\right]^2}\left|1-e^{\pm 2\pi\ii x}\right|^2\,d\mu_\xi(x,y)
&= 2-2\mathfrak{R}\left(\int\limits_{\left(-\frac{1}{2},\frac{1}{2}\right]^2}e^{\pm 2\pi\ii x}\,d\mu_\xi(x,y)\right)\\
&= \|\xi\|^2+\|\pi(\pm e_1)\xi\|^2-2\mathfrak{R}\left(\langle\pi(\pm e_1)\xi,\xi\rangle\right)\\
&= \|\pi(\pm e_1)\xi-\xi\|^2\\
&< \varepsilon^2.
\end{align*}
We warn the reader that we have used the same notation for the number $\pi$ ($e^{\pi\ii}=-1$), and for the representation $\pi:\F_2\ltimes\Z^2\to\mathbf{U}(\mathcal{H})$, hoping that the difference will be clear from the context. Now, if $\frac{1}{6}\leq|x|\leq\frac{1}{2}$,
\begin{align*}
\left|1-e^{\pm 2\pi\ii x}\right|^2 &= (1-\cos(2\pi x))^2+\sin(2\pi x)^2\\
&=2-2\cos(2\pi x)\\
&=4\sin(\pi x)^2\\
&\geq 1,
\end{align*}
which shows that
\begin{align*}
\mu_\xi\left(\left\{(x,y)\in \big(-\tfrac{1}{2},\tfrac{1}{2}\big]^2\ \mid\ |x|\geq\tfrac{1}{6}\right\}\right)
\leq \int\limits_{\left(-\frac{1}{2},\frac{1}{2}\right]^2}\left|1-e^{\pm 2\pi\ii x}\right|^2\,d\mu_\xi(x,y) < \varepsilon^2.
\end{align*}
Similarly, we find
\begin{align*}
\mu_\xi\left(\left\{(x,y)\in \big(-\tfrac{1}{2},\tfrac{1}{2}\big]^2\ \mid\ |y|\geq\tfrac{1}{6}\right\}\right)
 < \varepsilon^2.
\end{align*}
This shows that
\begin{align}\label{mu_xi(Omega)}
\mu_\xi(\Omega)>1-2\varepsilon^2.
\end{align}
On the other hand, by Proposition \ref{Prop_pvm}.(b), if $g\in\left\{R^{\pm}, T^{\pm}\right\}$ and $B\subseteq\big(-\tfrac{1}{2},\tfrac{1}{2}\big]^2$ is any Borel subset,
\begin{align*}
\left|\mu_\xi\left(g^\intercal B\right)-\mu_\xi(B)\right|
&= \left|\langle\pi(g)^{-1}\mathrm{P}(B)\pi(g)\xi,\xi\rangle-\langle\mathrm{P}(B)\xi,\xi\rangle\right|\\
&\leq \left|\langle\mathrm{P}(B)\pi(g)\xi,\pi(g)\xi-\xi\rangle\right|+\left|\langle\mathrm{P}(B)(\pi(g)\xi-\xi),\xi\rangle\right|\\
&\leq 2\|\pi(g)\xi-\xi\|\\
&< 2\varepsilon.
\end{align*}
Since the set $\left\{R^{\pm}, T^{\pm}\right\}$ is stable under taking transpose, we get
\begin{align}\label{mu_xig-mu_xi}
\left|\mu_\xi\left(g B\right)-\mu_\xi(B)\right|< 2\varepsilon,
\end{align}
for every $g\in\left\{R^{\pm}, T^{\pm}\right\}$. Now define a new measure $\mu$ on $\big(-\tfrac{1}{2},\tfrac{1}{2}\big]^2$ by
\begin{align}\label{mu=mu_xi_cap}
\mu(B)=\mu_\xi(B\cap\Omega).
\end{align}
By \eqref{mu_xig-mu_xi} and \eqref{mu_xi(Omega)}, for every Borel subset $B\subseteq\big(-\tfrac{1}{2},\tfrac{1}{2}\big]^2$ and every $g\in\left\{R^{\pm}, T^{\pm}\right\}$,
\begin{align*}
\mu(gB)-\mu(B) &= \mu(gB)-\mu_\xi(gB)+\mu_\xi(gB)-\mu_\xi(B)+\mu_\xi(B)-\mu(B)\\
&< 0+2\varepsilon+2\varepsilon^2\\
&= 2(\varepsilon+\varepsilon^2).
\end{align*}
Since this holds for every Borel subset, we have
\begin{align*}
|\mu(gB)-\mu(B)| < 2(\varepsilon+\varepsilon^2).
\end{align*}
Now, defining a new probability measure $\nu$ by
\begin{align*}
\nu(B)=\frac{\mu(B)}{\mu(\Omega)},
\end{align*}
we get, for every Borel subset $B\subseteq\big(-\tfrac{1}{2},\tfrac{1}{2}\big]^2$ and every $g\in\left\{R^{\pm}, T^{\pm}\right\}$,
\begin{align*}
|\nu(gB)-\nu(B)| < \frac{2(\varepsilon+\varepsilon^2)}{1-2\varepsilon^2}.
\end{align*}
Finally, by \eqref{mu_xi(0)} and \eqref{mu=mu_xi_cap}, $\nu$ is supported on $\big(-\frac{1}{6},\frac{1}{6}\big)^2\setminus\{(0,0)\}$.
\end{proof}

\begin{rmk}\label{Rmk_supp_nu}
The reason why it is important that the measure constructed in Lemma \ref{Lem_exist_pm} is supported on $\Omega=\big(-\frac{1}{6},\frac{1}{6}\big)^2$ is that, for every $g\in\left\{R^{\pm}, T^{\pm}\right\}$,
\begin{align*}
g\Omega\subseteq \left(-\tfrac{1}{2},\tfrac{1}{2}\right)^2
\end{align*}
for the action of $\F_2$ on $\R^2$. Hence, if we now view $\nu$ as a measure on $\R^2$ whose support is in $\Omega$, then the estimate \eqref{nu_almost_inv} is still valid for any Borel subset $B\subseteq\R^2$. Indeed, let $q:\R^2\to\big(-\tfrac{1}{2},\tfrac{1}{2}\big]^2$ denote the quotient map, and let $\tilde{\nu}$ be the extension of $\nu$ to $\R^2$ by $0$. Then
\begin{align*}
\tilde{\nu}(B)=\nu(q(B\cap\Omega))
\end{align*}
for every Borel subset $B\subseteq\R^2$. Moreover, if $\sigma$ denotes the induced action of $\F_2$ on the quotient,
\begin{align*}
\tilde{\nu}(gB)=\nu(q(gB\cap\Omega))=\nu\left(\sigma(g)q\left(B\cap g^{-1}\Omega\right)\right),
\end{align*}
for every $g\in\F_2$. On the other hand,
\begin{align*}
\tilde{\nu}(B) \geq \tilde{\nu}\left(B\cap g^{-1}\Omega\right) 
= \nu\left(q\left(B\cap g^{-1}\Omega\cap \Omega\right)\right).
\end{align*}
Therefore, if $g\in\left\{R^{\pm}, T^{\pm}\right\}$, then
\begin{align*}
\tilde{\nu}(B) \geq \nu\left(q\left(B\cap g^{-1}\Omega\right)\right),
\end{align*}
because $g^{-1}\Omega\subseteq \big(-\tfrac{1}{2},\tfrac{1}{2}\big)^2$ and $\nu$ is supported on $q(\Omega)$. Hence, by Lemma \ref{Lem_exist_pm},
\begin{align*}
\tilde{\nu}(gB)-\tilde{\nu}(B) 
&\leq \nu\left(\sigma(g)q\left(B\cap g^{-1}\Omega\right)\right) - \nu\left(q\left(B\cap g^{-1}\Omega\right)\right)\\
&< \frac{2(\varepsilon+\varepsilon^2)}{1-2\varepsilon^2}.
\end{align*}
Since this holds for every Borel subset $B\subseteq\R^2$ and every $g\in\left\{R^{\pm}, T^{\pm}\right\}$, we conclude that
\begin{align*}
\left|\tilde{\nu}(gB)-\tilde{\nu}(B)\right|
< \frac{2(\varepsilon+\varepsilon^2)}{1-2\varepsilon^2}.
\end{align*}
\end{rmk}

Now we prove a general result for measures on $\R^2$. This will give us a restriction on how small $\varepsilon$ can be in the context of Lemma \ref{Lem_exist_pm}.

\begin{lem}\label{Lem_nu_1/4}
Let $\nu$ be a Borel probability measure on $\R^2\setminus\{(0,0)\}$. Then there are $g\in\left\{R^{\pm}, T^{\pm}\right\}$ and a Borel subset $W\subseteq\R^2\setminus\{(0,0)\}$ such that
\begin{align*}
|\nu(gW)-\nu(W)| \geq \frac{1}{4},
\end{align*}
where $R^{\pm}, T^{\pm}$ are defined as in \eqref{R+-T+-}.
\end{lem}
\begin{proof}
Assume by contradiction that
\begin{align}\label{nu<1/4}
|\nu(gW)-\nu(W)| < \frac{1}{4},
\end{align}
for every Borel subset $W\subseteq\R^2\setminus\{(0,0)\}$ and every $g\in\left\{R^{\pm}, T^{\pm}\right\}$. Let us divide $\R^2\setminus\{(0,0)\}$ into 6 regions whose boundaries are given by the lines $y=0$, $y=\frac{1}{2}x$, $y=2x$, $x=0$, $y=-2x$, $y=-\frac{1}{2}x$. We let $A$ be the region comprised between the lines $y=0$ (included) and $y=\frac{1}{2}x$ (not included). We define $B,C,D,E,F$ analogously, following the counter-clockwise order; see Figure \ref{Fig_plane6}. We have
\begin{align*}
\R^2\setminus\{(0,0)\}=A\sqcup B\sqcup C\sqcup D\sqcup E\sqcup F.
\end{align*}
\begin{figure}[h]
\includegraphics[scale=1.5]{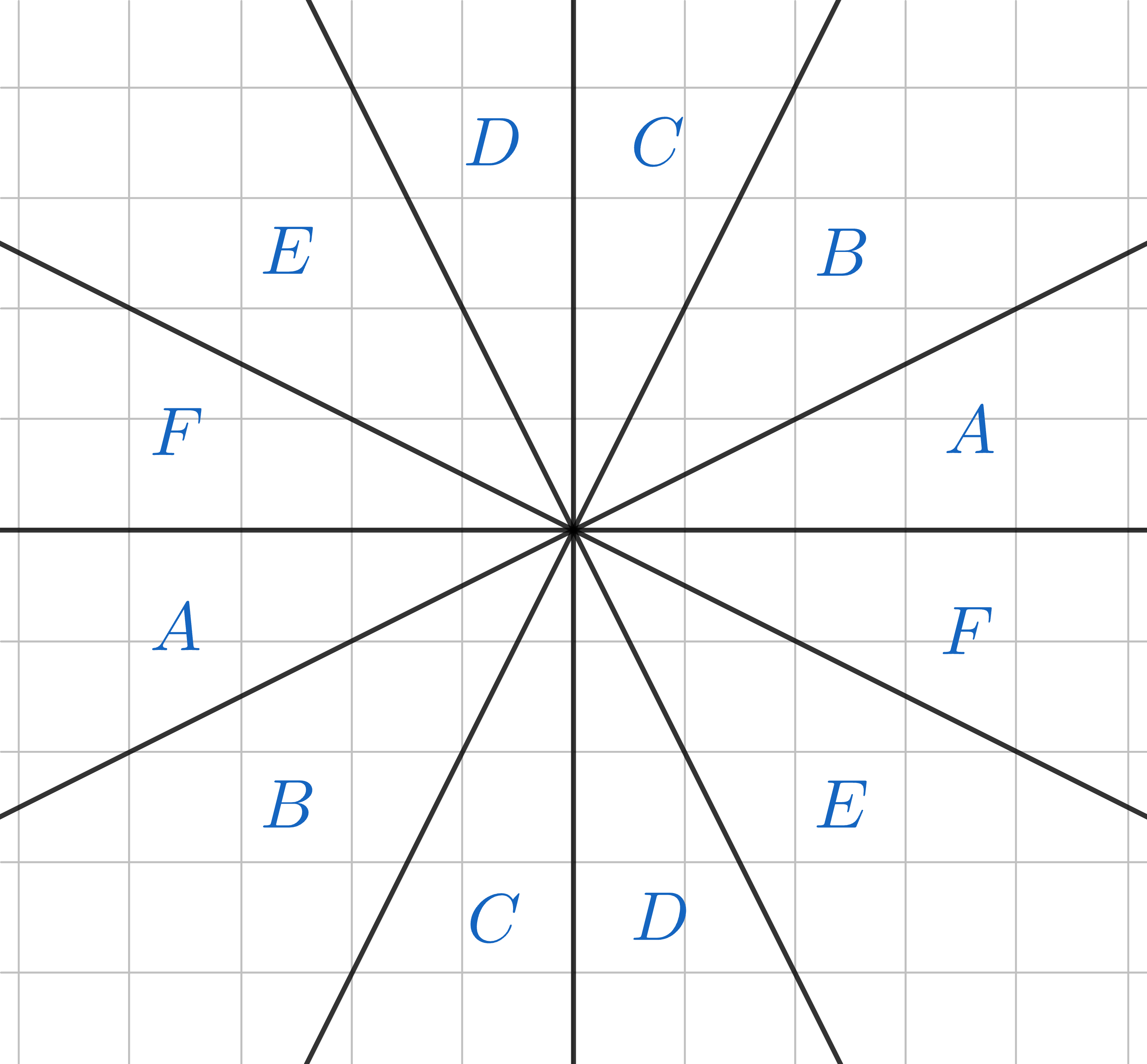}
\caption{Partition of $\R^2\setminus\{(0,0)\}$ into the regions $A,B,C,D,E,F$.}
\label{Fig_plane6}
\end{figure}

\noindent Observe that
\begin{align*}
R^+(A\cup B\cup C)&=A, & T^+(A\cup B\cup C)&=C,\\
R^{-}(D\cup E\cup F)&=F, & T^{-}(D\cup E\cup F)&=D.
\end{align*}
By \eqref{nu<1/4}, together with the identities above,
\begin{align*}
\nu(B)+\nu(C)&<\frac{1}{4}, & \nu(A)+\nu(B)&<\frac{1}{4},\\
\nu(D)+\nu(E)&<\frac{1}{4}, & \nu(E)+\nu(F)&<\frac{1}{4}.
\end{align*}
Summing all these inequalities, and using the fact that
\begin{align*}
\nu(A)+\nu(B)+\nu(C)+\nu(D)+\nu(E)+\nu(F)=1,
\end{align*}
we obtain
\begin{align*}
1+\nu(B)+\nu(E)<1,
\end{align*}
which is impossible.
\end{proof}

Now we combine Lemmas \ref{Lem_exist_pm} and \ref{Lem_nu_1/4} in order to find a lower bound for the Kazhdan constant of the pair $(\F_2\ltimes\Z^2,\Z^2)$ for the generating set \eqref{gen_sset_sdp}.

\begin{cor}\label{Cor_Kazh_sdp}
Let $S$ be the subset of $\F_2\ltimes\Z^2$ defined in \eqref{gen_sset_sdp}. Then
\begin{align*}
\kappa(\F_2\ltimes\Z^2,\Z^2,S)\geq \frac{\sqrt{26}-4}{10}\quad (\approx 0.11),
\end{align*}
where $\kappa$ is defined as in \eqref{Kazh_const}.
\end{cor}
\begin{proof}
Assume by contradiction that
\begin{align*}
\kappa(\F_2\ltimes\Z^2,\Z^2,S) < \frac{\sqrt{26}-4}{10}.
\end{align*}
Then there exist $\varepsilon<\frac{\sqrt{26}-4}{10}$, a unitary representation $\pi:\F_2\ltimes\Z^2\to\mathbf{U}(\mathcal{H})$ without nontrivial $\Z^2$-invariant vectors, and a unit vector $\xi\in\mathcal{H}$ such that
\begin{align*}
\max_{g\in S}\|\pi(g)\xi-\xi\|<\varepsilon.
\end{align*}
By Lemma \ref{Lem_exist_pm}, there is a Borel probability measure $\nu$ on $\big(-\tfrac{1}{2},\tfrac{1}{2}\big]^2$, supported on ${\big(-\tfrac{1}{6},\tfrac{1}{6}\big)^2}\setminus\{(0,0)\}$, such that, for every Borel set $B\subseteq\big(-\tfrac{1}{2},\tfrac{1}{2}\big]^2$ and every $g\in\left\{R^{\pm}, T^{\pm}\right\}$,
\begin{align*}
|\nu(gB)-\nu(B)| < \frac{2(\varepsilon+\varepsilon^2)}{1-2\varepsilon^2}.
\end{align*}
Extending $\nu$ by $0$, we may view it as a measure on $\R^2\setminus\{(0,0)\}$; see Remark \ref{Rmk_supp_nu}. Hence, by Lemma \ref{Lem_nu_1/4}, we have
\begin{align*}
\frac{2(\varepsilon+\varepsilon^2)}{1-2\varepsilon^2}>\frac{1}{4},
\end{align*}
which is equivalent to
\begin{align*}
10\varepsilon^2+8\varepsilon-1> 0.
\end{align*}
Since the roots of the polynomial $10x^2+8x-1$ are
\begin{align*}
\frac{-\sqrt{26}-4}{10}\qquad\text{and}\qquad\frac{\sqrt{26}-4}{10},
\end{align*}
this cannot happen for $\varepsilon<\frac{\sqrt{26}-4}{10}$. We conclude that
\begin{align*}
\kappa(\F_2\ltimes\Z^2,\Z^2,S)\geq \frac{\sqrt{26}-4}{10}.&\qedhere
\end{align*}
\end{proof}

Now we have all the necessary ingredients for the proof of Theorem \ref{Thm_F2xZ2}.

\begin{proof}[Proof of Theorem \ref{Thm_F2xZ2}]
Let $\sigma:\F_2\ltimes\Z^2\to\operatorname{BiLip}_+^{\mathrm{bd}}(\R)$ be an injective homomorphism such that $\sigma(\Z^2)$ does not have global fixed points. Let $S=\{R,T,e_1,e_2\}$ be the generating set defined in \eqref{gen_set_intro}. Observe that $\tilde{S}=S\cup S^{-1}$ is exactly the symmetric generating set \eqref{gen_sset_sdp}. By Theorem \ref{Thm_T_biLip},
\begin{align*}
\max_{g\in \tilde{S}}\operatorname{BiLip}(\sigma(g))\geq \Phi\left(\kappa\left(\F_2\ltimes\Z^2,\Z^2,\tilde{S}\right)\right),
\end{align*}
where $\Phi$ is defined as in \eqref{def_Phi}. Now observe that
\begin{align*}
\max_{g\in \tilde{S}}\operatorname{BiLip}(\sigma(g))=\max_{g\in S}\operatorname{BiLip}(\sigma(g))
\end{align*}
simply because
\begin{align*}
\operatorname{BiLip}(\sigma(g^{-1}))=\operatorname{BiLip}(\sigma(g)).
\end{align*}
%
%
%
Hence, by Corollary \ref{Cor_Kazh_sdp}, together with the fact that $\Phi$ is increasing,
\begin{align*}
\max_{g\in S}\operatorname{BiLip}(\sigma(g))\geq \Phi\left(\frac{\sqrt{26}-4}{10}\right)
= \exp\left(\frac{\sqrt{26}-4}{5}\right).&\qedhere
\end{align*}
\end{proof}

\section{Consequence for orderable groups}\label{S_ord}
We end this paper with the proof of Corollary \ref{Cor_ord_Kazh}. The main ingredient is the following result, which was essentially proved in \cite[Theorem 8.5]{DKNP}.

\begin{thm}[Deroin--Kleptsyn--Navas--Parwani]\label{Thm_ord_BiLp}
Let $G$ be a finitely generated, orderable group. Let $S$ be a finite, symmetric generating subset of $G$. Then $G$ is isomorphic to a subgroup of $\operatorname{BiLip}_+^{\mathrm{bd}}(\R)$ without global fixed points. Moreover, for every $g\in S$,
\begin{align}\label{Lip(g)<S}
\operatorname{BiLip}(g)\leq |S|.
\end{align}
\end{thm}
\begin{proof}
By \cite[Proposition 1.1.8]{DeNaRi}, we can identify $G$ with a subgroup of $\operatorname{Homeo}_+(\R)$ without global fixed points. By \cite[Theorem 8.5]{DKNP}, it can be conjugated to a subgroup of $\operatorname{BiLip}_+^{\mathrm{bd}}(\R)$. Moreover, by considering the probability measure
\begin{align*}
\mu(g)=\begin{cases}
\frac{1}{|S|}& \text{if }g\in S,\\
0& \text{otherwise},
\end{cases}
\end{align*}
the proof of \cite[Theorem 8.5]{DKNP} produces an action satisfying
\begin{align*}
\operatorname{BiLip}(g)\leq\frac{1}{\mu(g)}=|S|
\end{align*}
for every $g\in S$.
\end{proof}

With this, we can obtain bounds for the Kazhdan constants of orderable groups.

\begin{proof}[Proof of Corollary \ref{Cor_ord_Kazh}]
Let $G$ be a finitely generated orderable group, $\Gamma$ a subgroup of $G$, and $S$ a finite, symmetric generating subset of $G$. By Theorem \ref{Thm_ord_BiLp}, $G$ is isomorphic to a subgroup of $\operatorname{BiLip}_+^{\mathrm{bd}}(\R)$ without global fixed points and such that $\operatorname{BiLip}(g)\leq |S|$ for every $g\in S$. Hence, by Theorem \ref{Thm_T_biLip},
\begin{align*}
\kappa(G,\Gamma,S)\leq\Phi^{-1}(|S|).&\qedhere
\end{align*}
\end{proof}





\bibliographystyle{plain} 

\bibliography{Bibliography}

\end{document}